\documentclass{amsart}
\usepackage{mathptmx}
\usepackage{amssymb}
\usepackage{amsfonts}
\usepackage{amsmath}
\usepackage{graphicx}
\usepackage{shadow}
\usepackage{color}
\usepackage[all]{xy}
\usepackage[pagebackref]{hyperref}
\newtheorem{thm}{Theorem}[section]
\newtheorem{corollary}[thm]{Corollary}
\newtheorem{lemma}[thm]{Lemma}

\newtheorem{claim}{Claim}

\theoremstyle{definition}
\newtheorem{definition}[thm]{Definition}
\newtheorem{example}[thm]{Example}

\theoremstyle{remark}
\newtheorem{remark}[thm]{Remark}

\newcommand{\field}[1]{\mathbb{#1}}

\newcommand{\Q }{\field{Q}}
\newcommand{\Z }{\field{Z}}

\DeclareMathOperator{\qf}{qf}
\DeclareMathOperator{\td}{t.d.}

\DeclareMathOperator{\Spec}{Spec}

\DeclareMathOperator{\m}{\frak{m}}
\DeclareMathOperator{\edim}{embdim}
\DeclareMathOperator{\charac}{char}

\DeclareMathOperator{\Sup}{Sup}
\DeclareMathOperator{\Tor}{Tor}

\begin{document}
%%%%%%%%%%%%%%%%%%%%%%%%%%%%%%%%%%%%%%%%%%%%%%%%%%%%%%%%%%%%%%%%%%%%%%%%%%%%%%%%%%%%%%%%%%%%%%%%%%%%%%%%%%%%%%%%%%%%%%%%%%%%%%%%%%%%%%%%%%%%%%
%%%%%%%%%%%%%%%%%%%%%%%%%%%%%%%%%%%%%%%%%%%%%%%%%%%%%%%%%%%%%%%%%%%%%%%%%%%%%%%%%%%%%%%%%%%%%%%%%%%%%%%%%%%%%%%%%%%%%%%%%%%%%%%%%%%%%%%%%%%%%%
%%%%%%%%%%%%%%%%%%%%%%%%%%%%%%%%%%%%%%%%%%%%%%%%%%%%%%%%%%%%%%%%%%%%%%%%%%%%%%%%%%%%%%%%%%%%%%%%%%%%%%%%%%%%%%%%%%%%%%%%%%%%%%%%%%%%%%%%%%%%%%
%%%%%%%%%%%%%%%%%%%%%%%%%%%%%%%%%%%%%%%%%%%%%%%%%%%%%%%%%%%%%%%%%%%%%%%%%%%%%%%%%%%%%%%%%%%%%%%%%%%%%%%%%%%%%%%%%%%%%%%%%%%%%%%%%%%%%%%%%%%%%%

\title[Regularity of tensor products of $k$-algebras]{Regularity of tensor products of $k$-algebras $^{(\star)}$}
\thanks{$^{(\star)}$ This work was supported by KFUPM under DSR Research Grant \# FT100021.}

\author{S. Bouchiba}
\address{Department of Mathematics, University of Meknes, Meknes 50000, Morocco}
\email{bouchiba@fs-umi.ac.ma}

\author[S. Kabbaj]{S. Kabbaj $^{(1)}$}\thanks{$^{(1)}$ Corresponding author.}
\address{Department of Mathematics and Statistics, KFUPM, Dhahran 31261, KSA}
\email{kabbaj@kfupm.edu.sa}

\date{\today}


%\dedicatory{}

%%%%%%%%%%%%%%%%%%%%%%%%%%%%%%%%%%%%%%%%%%%%%%%%%%%%%%%%%%%%%%%%%%%%%%%%%%%%%%%%%%%%%%%%%%%%%%%%%%%%%%%%%%%%%%%%%%%%%%%%%%%%%%%%%%%%%%%%%%%%%%
%%%%%%%%%%%%%%%%%%%%%%%%%%%%%%%%%%%%%%%%%%%%%%%%%%%%%%%%%%%%%%%%%%%%%%%%%%%%%%%%%%%%%%%%%%%%%%%%%%%%%%%%%%%%%%%%%%%%%%%%%%%%%%%%%%%%%%%%%%%%%%
%%%%%%%%%%%%%%%%%%%%%%%%%%%%%%%%%%%%%%%%%%%%%%%%%%%%%%%%%%%%%%%%%%%%%%%%%%%%%%%%%%%%%%%%%%%%%%%%%%%%%%%%%%%%%%%%%%%%%%%%%%%%%%%%%%%%%%%%%%%%%%
%%%%%%%%%%%%%%%%%%%%%%%%%%%%%%%%%%%%%%%%%%%%%%%%%%%%%%%%%%%%%%%%%%%%%%%%%%%%%%%%%%%%%%%%%%%%%%%%%%%%%%%%%%%%%%%%%%%%%%%%%%%%%%%%%%%%%%%%%%%%%%
\begin{abstract}
This paper tackles a problem on the possible transfer of regularity to tensor products of algebras over a field $k$. The main result establishes necessary and sufficient conditions for a Noetherian tensor product of two extension fields of $k$ to inherit regularity in various settings of separability. Thereby, we provide some applications as well as several original examples to illustrate or delimit the scope of the established results.
\end{abstract}
\maketitle

%%%%%%%%%%%%%%%%%%%%%%%%%%%%%%%%%%%%%%%%%%%%%%%%%%%%%%%%%%%%%%%%%%%%%%%%%%%%%%%%%%%%%%%%%%%%%%%%%%%%%%%%%%%%%%%%%%%%%%%%%%%%%%%%%%%%%%%%%%%%%%
%%%%%%%%%%%%%%%%%%%%%%%%%%%%%%%%%%%%%%%%%%%%%%%%%%%%%%%%%%%%%%%%%%%%%%%%%%%%%%%%%%%%%%%%%%%%%%%%%%%%%%%%%%%%%%%%%%%%%%%%%%%%%%%%%%%%%%%%%%%%%%
%%%%%%%%%%%%%%%%%%%%%%%%%%%%%%%%%%%%%%%%%%%%%%%%%%%%%%%%%%%%%%%%%%%%%%%%%%%%%%%%%%%%%%%%%%%%%%%%%%%%%%%%%%%%%%%%%%%%%%%%%%%%%%%%%%%%%%%%%%%%%%
%%%%%%%%%%%%%%%%%%%%%%%%%%%%%%%%%%%%%%%%%%%%%%%%%%%%%%%%%%%%%%%%%%%%%%%%%%%%%%%%%%%%%%%%%%%%%%%%%%%%%%%%%%%%%%%%%%%%%%%%%%%%%%%%%%%%%%%%%%%%%%
\begin{section}{Introduction}

\noindent All algebras considered are commutative with identity elements and, unless otherwise specified, are assumed to be non-trivial. All ring homomorphisms are unital. Throughout, $k$ stands for a field. A Noetherian local ring $(R, \m)$ is regular if its Krull and embedding dimensions coincide; i.e., $\dim(R) = \edim(R)$, where $\edim(R)$ denotes the dimension of $\frac{\m}{\m^{2}}$ as an $\frac{R}{\m}$-vector space. Regular local rings were first introduced by Krull, and then became prominent once Zariski showed that, geometrically, a regular local ring corresponds to a smooth point on an algebraic variety. Later, Serre found a homological characterization for a local ring $R$ to be regular; that is, $R$ has finite global dimension. Finite global dimension is preserved under localization, so that localizations of regular local rings at prime ideals are again regular. Geometrically, this corresponds to the intuition that if a surface contains a smooth curve, then the surface is smooth near the curve. Consequently, the definition of regularity got globalized as follows: A Noetherian ring $R$ is regular if its localizations with respect to all prime ideals are regular. Using homological techniques, Auslander and Buchsbaum proved in 1950's that every regular local ring is a UFD.

A  Noetherian local ring $(R,\m)$ is a complete intersection if the completion $\hat{R}$ of $R$ with respect to the $\m$-adic topology is the quotient ring of a regular local ring modulo an ideal generated by a regular sequence. The ring $R$ is Gorenstein if its injective dimension (as an $R$-module) is finite; and $R$ is Cohen-Macaulay if grade and height coincide for every ideal of $R$. These notions are globalized by carrying over to localizations with respect to the prime ideals. We have the following diagram of implications:
\bigskip

\begin{center}
Regular ring\\
$\Downarrow$ \\
(Locally) Complete Intersection ring\\
$\Downarrow$ \\
Gorenstein ring\\
$\Downarrow$ \\
Cohen-Macaulay ring\\
$\Downarrow$ \\
Noetherian ring
\end{center}
\bigskip

In this paper we will tackle a problem, originally initiated by Grothendieck \cite{Gr}, on the possible transfer of regularity to tensor products of k-algebras. Recently, it has been proved that a Noetherian tensor product of $k$-algebras $A\otimes_{k}B$ inherits from $A$ and $B$ the notions of  locally complete intersection ring, Gorenstein ring, and Cohen-Macaulay ring \cite{BK,HTY, S, S3,TY}. In particular, $K\otimes_{k}L$ is a locally complete intersection ring, for any two extension fields $K$ and $L$ of $k$ such that $K\otimes_{k}L$ is Noetherian \cite[Proposition 5]{TY}. Notice at this point that tensor products of rings subject to the above concepts were recently used to broaden or delimit the context of validity of some homological conjectures; see for instance \cite{HJ,J}.

As to regularity, the problem remains elusively open. Indeed, contrary to the above notions, a Noetherian tensor product of two extension fields of $k$ is not regular in general. In 1965, Grothendieck proved that $K\otimes_{k}L$ is a regular ring provided $K$ or $L$ is a finitely generated separable  extension field of $k$ \cite[Lemma 6.7.4.1]{Gr}. In 1969, Watanabe, Ishikawa, Tachibana, and Otsuka, showed that under a suitable condition tensor products of regular rings are complete intersections \cite[Theorem 2, p. 417]{WITO}.  In 2003, Tousi and Yassemi proved that a Noetherian tensor product of two $k$-algebras $A$ and $B$ is regular if and only if so are $A$ and $B$ in the special case where $k$ is perfect; i.e., every (algebraic) extension of $k$ is separable \cite{TY, HTY}.

Recall that regularity, though a topic of commutative Noetherian rings, proved to be well approached via homological methods. In fact, a characterization of regular homomorphisms
$R\longrightarrow S$ is given by the vanishing of the first Andr\'e-Quillen homology functor $D_{1}(S/R,-)$. In the case of a homomorphism of fields $k\longrightarrow K$, the vanishing of $D_{1}(K/k,-)$ totally characterizes separability of $K$ over $k$. So that, under separability and Noetherianity, $K\otimes_kA$ inherits regularity via base change.
Nevertheless, the case of tensor products of $k$-algebras involving purely inseparable extensions of $k$ remains unexplored. The main goal of this paper is to handle such a case. Actually, our main result (Theorem~\ref{Reg:2}) establishes necessary and sufficient conditions for a Noetherian tensor product of two extension fields of $k$ to inherit regularity; and hence generalizes Grothendieck's aforementioned result. As a prelude to this, we revisit the constructions of the form $A\otimes_{k}B$ where  $A$ or $B$ is geometrically regular (Lemma~\ref{Reg:1}) and then offer a new direct proof (without use of Andr\'e-Quillen homology). We close with a discussion of the correlation between $A\otimes_kB$ and its fiber rings when subject to regularity. It turns out that, in case $A$ (or $B$) is assumed to be residually separable, $A\otimes_{k}B$ is regular if and only if so are $A$ and $B$ (Theorem~\ref{Reg:3}). This is a slight improvement of \cite[Theorem 6(c)]{TY}. All along the paper, several original examples are provided to illustrate or delimit the scope of the established results.
\end{section}

%%%%%%%%%%%%%%%%%%%%%%%%%%%%%%%%%%%%%%%%%%%%%%%%%%%%%%%%%%%%%%%%%%%%%%%%%%%%%%%%%%%%%%%%%%%%%%%%%%%%%%%%%%%%%%%%%%%%%%%%%%%%%%%%%%%%%%%%%%%%%%
%%%%%%%%%%%%%%%%%%%%%%%%%%%%%%%%%%%%%%%%%%%%%%%%%%%%%%%%%%%%%%%%%%%%%%%%%%%%%%%%%%%%%%%%%%%%%%%%%%%%%%%%%%%%%%%%%%%%%%%%%%%%%%%%%%%%%%%%%%%%%%
%%%%%%%%%%%%%%%%%%%%%%%%%%%%%%%%%%%%%%%%%%%%%%%%%%%%%%%%%%%%%%%%%%%%%%%%%%%%%%%%%%%%%%%%%%%%%%%%%%%%%%%%%%%%%%%%%%%%%%%%%%%%%%%%%%%%%%%%%%%%%%
%%%%%%%%%%%%%%%%%%%%%%%%%%%%%%%%%%%%%%%%%%%%%%%%%%%%%%%%%%%%%%%%%%%%%%%%%%%%%%%%%%%%%%%%%%%%%%%%%%%%%%%%%%%%%%%%%%%%%%%%%%%%%%%%%%%%%%%%%%%%%%
\begin{section}{Transfer of regularity to tensor products of $k$-algebras}\label{Reg}

A transcendence base $B$ of an extension field $K$ over $k$ is called a separating transcendence base if $K$ is separable algebraic over $k(B)$; and $K$ is said to be separable over $k$ if every finitely generated intermediate field has a separating transcendence base over $k$. Finally, recall that a Noetherian ring $A$ containing a field $k$ is said to be geometrically regular over $k$ if $A\otimes_{k} F$ is a regular ring for every finite extension $F$ of $k$; and a homomorphism $\varphi\colon A\rightarrow B$ of Noetherian rings is said to be regular if $\varphi$ is flat and $B\otimes_{A}\kappa_{A}(p)$ is geometrically regular over $\kappa_{A}(p)$ for each $p\in\Spec(A)$, where $\kappa_{A}(p)$ denotes the residue field of $A_{p}$ \cite[\S 32, pp. 255-256]{Mat}.

In 1965, Grothendieck proved that if $K$ and $L$ are two extension fields of $k$ such that either $K$ or $L$ is finitely generated over $k$ and if $K$ is separable over $k$, then  $K\otimes_{k}L$ is regular \cite[Lemma 6.7.4.1]{Gr}. More generally, if  $K$ is a separable extension field of $k$ and $A$ is a regular finitely generated $k$-algebra, then $K\otimes_kA$ is regular; indeed, separability implies that $k\rightarrow K$ is regular. Then a base change via the finite type homomorphism $k\rightarrow A$ yields that $A\rightarrow K\otimes_kA$ is regular since regularity of the fibers is preserved (as the residue fields of $A$ are finitely generated extensions of $k$). By \cite[Theorem 32.2(i)]{Mat}, $K\otimes_kA$ is regular.

 Now, let us substitute the assumption ``$K\otimes_kA$ is Noetherian" for ``$A$ is a finitely generated $k$-algebra." In this case, regularity is transferred to $K\otimes_kA$ through base change of regular homomorphisms via Andr\'e-Quillen homology (which requires no finite type assumption). Indeed, by \cite[(6.3)]{Iy}, $D_n(K\otimes A/A, -)\cong D_n(K/k, -)$ for every $n\in \Z$ (since we are in the trivial case where $\Tor_{n}^{k}(K,A) = 0$ for every $n\geq1$). Then, by \cite[Theorem 9.5]{Iy}, $A\rightarrow K\otimes A$ is regular if and only if $k\rightarrow K$ is regular. So that, under separability and Noetherianity,    $K\otimes_kA$ is regular if and only if so is $A$. We were not able to locate any explicit reference for this result. Next we record this fact in a slightly more general form  and also offer a new direct proof (without use of Andr\'e-Quillen homology) via the prime ideal structure (Cf. \cite[Proposition 4.14]{BDK}).

%%%%%%%%%%%%%%%%%%%%%%%%%%%%%%%%%%%%%%%%%%%%%%%%%%%%%%%%%%%%%%%%%%
%%%%%%%%%%%%%%%%%%%%%%%%%%%%%%%%%%%%%%%%%%%%%%%%%%%%%%%%%%%%%%%%%%
\begin{lemma}\label{Reg:1}
Let $A$ and $B$ be two $k$-algebras such that $A$ is geometrically regular. Then the following assertions are equivalent:
\begin{enumerate}
\item $A\otimes_kB$ is regular;
\item $B$ is regular and $A\otimes_kB$ is Noetherian.
\end{enumerate}
\end{lemma}

%%%%%%%%%%%%%%%%%%%%%
%%%%%%%%%%%%%%%%%%%%%
\begin{proof}
The implication (i) $\Rightarrow$ (ii) is straightforward by \cite[Corollary 4]{TY}. Next, we prove (ii) $\Rightarrow$ (i) via two steps.

{\em Step \rm1.} Suppose that $B=K$ is an extension field of $k$ such that $A\otimes_kK$ is Noetherian.  Let $\Delta$ denote the set of all finitely generated extension
fields of $k$ contained in $K$ and let  $$D:=A\otimes_kK=\lim_{\substack{\rightarrow\\ E\in \Delta}}D(E)$$ where $D(E):=A\otimes_kE$ for each $E\in \Delta$. Fix a prime ideal $P$ of
$D$ and let $P_E:=P\cap D(E)$ for each $E\in \Delta$.

\begin{claim} If $F\in \Delta$ such that $P_ED=P_FD$ for each $E\in
\Delta$ containing $F$, then $P=P_FD$.\end{claim}

In fact, let $F\in \Delta$ such that $P_ED=P_FD$ for each $E\in
\Delta$ containing $F$. Let $x\in P$. Then there exists
$E^{\prime}\in \Delta$ such that $x\in D(E^{\prime})$, and thus
$x\in P_{E^{\prime}}D$. Whence, $x\in P_{F(E^{\prime})}=P_FD$, where
$F(E^{\prime})$ denotes the composite field of $F$ and $E^{\prime}$
in $K$. It follows that $P=P_FD$, proving the claim.

\begin{claim} There exists $E\in \Delta$ such that $P=P_ED$.\end{claim}

Assume, by way of contradiction, that $P_ED\subsetneqq P$ for any $E\in
\Delta$ (notice that under this hypothesis $K$ is necessarily infinitely generated over $k$; i.e., $K\notin \Delta$). Choose $E_1\in \Delta$. By Claim 1, there exists $E_2\in \Delta$ containing
$E_1$ such that $P_{E_1}D\subsetneqq P_{E_2}D$. Iterating this
process yields the following infinite chain of ideals in $D$
$$P_{E_1}D\subsetneqq P_{E_2}D\subsetneqq ...\subsetneqq P_{E_n}D\subsetneqq ...\subsetneqq P$$ where the $E_j\in \Delta$. This
leads to a contradiction since $D$ is Noetherian. Hence there exists
$E\in \Delta$ such that $P=P_ED$, as desired.

\begin{claim} $PD_P$ is generated by a $D_P$-regular sequence.\end{claim}

Indeed, by Claim 2, $P=P_ED$ for some $E\in\Delta$. Now, observe that
$$D_P:=(A\otimes_kK)_P\cong\Big (D(E)_{P_E}\otimes_EK\Big )_P\ \mbox
{ and }\ PD_P\cong \Big (P_ED(E)_{P_E}\otimes_EK\Big )D_P$$ with
$P_ED(E)_{P_E}$ being the maximal ideal of $D(E)_{P_E}$. As $E$
is finitely generated over $k$, $D(E)$ is regular (recall that $A$ is geometrically regular). Hence $D(E)_{P_E}$ is a regular local ring.
By \cite[Theorem 169]{Kap}, $P_ED(E)_{P_E}$ is generated by a
$D(E)_{P_E}$-regular sequence $x_1,x_2,...,x_r$. Further, it is
easily seen that $x_1\otimes_k1,x_2\otimes_E1,...,x_r\otimes_E1$ is
a $D(E)_{P_E}\otimes_EK$-regular sequence of $P_ED(E)_{P_E}\otimes_EK$. As $\Big (P_ED(E)_{P_E}\otimes_EK\Big
)D_P\cong PD_P$, we get, by \cite[Theorem 133]{Kap},
$\displaystyle
{\frac {x_1\otimes_E1}1,\frac {x_2\otimes_E1}1,...,\frac
{x_n\otimes_E1}1}$ is a $D_P$-regular
sequence of $PD_P$. Now, since
$P_ED(E)_{P_E}=(x_1,x_2,...,x_n)D(E)_{P_E}$, we get $$PD_P=\Big
(\displaystyle {\frac {x_1\otimes_E1}1,\frac
{x_2\otimes_E1}1,...,\frac {x_n\otimes_E1}1}\Big )D_P$$ establishing
the claim.

 It follows, by \cite[Theorem 160]{Kap}, that $D_P$ is a regular local
 ring. Consequently, $D$ is a regular ring, as desired.

{\em Step \rm2.}  Suppose that $B$ is a regular $k$-algebra such that $A\otimes_kB$
 is Noetherian. Let $q\in \Spec(B)$. First, as $A\otimes_kB$ is Noetherian,
$A\otimes_kk_B(q)$, being a localization of a quotient
 of $A\otimes_kB$, is Noetherian. Then, by Step 1,
 $A\otimes_kk_B(q)$ is regular for each $q\in\Spec(B)$. Now,
\cite[Corollary 2]{TY} yields that $A\otimes_kB$ is regular, completing the proof of the theorem.
\end{proof}
%%%%%%%%%%%%%%%%%%%%%
%%%%%%%%%%%%%%%%%%%%%

In particular, if $K$ is a separable extension field of $k$ and $A$ is a $k$-algebra, then $K\otimes_kA$ is regular if and only if $A$ is regular and $K\otimes_kA$ is Noetherian. Example~\ref{Reg:4} shows that this result is not true, in general, if one substitutes pure inseparability for separability; and that, however, this latter condition is not necessary.

Recall that if $K$ and $L$ are two extension fields of $k$ such that one of them is finitely generated, then $K\otimes_kL$ is Noetherian \cite{WITO}. The converse is not true in general; e.g., $$\Q(x_{1}, x_{2}, ... )\otimes \Q(\sqrt{2}, \sqrt{3}, ... )\cong \Q(\sqrt{2}, \sqrt{3}, ... )(x_{1}, x_{2}, ... )$$ is a field, where $x_{1}, x_{2}, ...$ are infinitely many indeterminates over $\Q$. However, the converse holds in the case $K=L$ \cite[Corollary 3.6]{Fer} or \cite[Theorem 11]{V}. These facts combined with Lemma~\ref{Reg:1} yield the following remark, where the separability assumption is required only for regularity.

%%%%%%%%%%%%%%%%%%%%%%%%%%%%%%%%%%%%%%%%%%%%%%%%%%%%%%%%%%%%%%%%%%
%%%%%%%%%%%%%%%%%%%%%%%%%%%%%%%%%%%%%%%%%%%%%%%%%%%%%%%%%%%%%%%%%%
\begin{remark}\label{Reg:1.1}
Let $K$ and $L$ be two extension fields of $k$ and assume that $K$ is separable over $k$. Then:
$K$ or $L$ is finitely generated $\Rightarrow $ $K\otimes_kL$ is Noetherian $\Leftrightarrow$ $K\otimes_kL$ is regular. The special case where $K=L$ is handled by Corollary~\ref{Reg:2.2}.
\end{remark}

For an arbitrary $k$-algebra $A$ (not necessarily a domain), the transcendence degree over $k$ is given by (cf. \cite[p. 392]{W}) $$\td(A:k):=\Sup\{\td(\dfrac Ap:k)\mid p\in \Spec(A)\}.$$ Further, if $A$ and $B$ are two $k$-algebras such that $A\otimes _kB$ is Noetherian, then necessarily $A$ and $B$ are  Noetherian rings and either $\td(A:k)<\infty$ or $\td(B:k)<\infty$  (cf. \cite[p. 69]{BK}). Also, for any two extension fields $K$ and $L$ of $k$, \cite[Theorem 3.1]{S2} asserts that $$\dim(K\otimes_kL)=\min\{\td(K:k),\td(L:k)\}.$$ These facts allow one to give illustrative examples of regular tensor products (of fields) of arbitrary dimension.

%%%%%%%%%%%%%%%%%%%%%%%%%%%%%%%%%%%%%%%%%%%%%%%%%%%%%%%%%%%%%%%%%%
%%%%%%%%%%%%%%%%%%%%%%%%%%%%%%%%%%%%%%%%%%%%%%%%%%%%%%%%%%%%%%%%%%
\begin{example}\label{Reg:1.2}\rm
Let $x_{1}, x_{2}, ...$ be infinitely many indeterminates over $k$. Then, for any positive integer $n$,
$k(x_{1}, ..., x_{n})\otimes k(x_{1}, x_{2}, ...)$ is an $n$-dimensional regular ring.
\end{example}

Note that $k(x_{1}, ..., x_{n})$ and $k(x_{1}, x_{2}, ...)$ are (non-algebraic) separable extensions of $k$ by Mac Lane's Criterion. For the algebraic separable case, see Example~\ref{Reg:2.5}.

Let $K$ and $L$ be two extension fields of $k$. Assume that $K$ is purely inseparable over $k$ and let $\overline{L}$ be an algebraic
closure of $L$. Then there exists a unique $k$-homomorphism $u:K\rightarrow \overline{L}$ \cite[Proposition 3, p. V.25]{B4-7}, and
the isomorphic image $u(K)$ is obviously purely inseparable over $k$. In this vein, we can always view
 $K$ and $L$ as subfields of a common field $\overline{L}$. Recall Mac Lane's notion of linear disjointness; namely, $K$ and $L$ are linearly disjoint over $k$ if every subset of $K$ which is linearly independent over $k$ is also linearly independent over $L$; equivalently, if $K\otimes_kL$ is a domain.

In the sequel, given an extension field $K$ of $k$, $K_s$ and $K_i$ will denote the (not necessarily algebraic) separable closure and (algebraic) purely inseparable  closure
of $k$ in $K$, respectively. Notice that $K$ is an extension field of the composite field $K_{s}K_{i}$ and the equality $K_{s}K_{i}= K$ holds, for instance, when $K$ is separable, purely inseparable, or normal over $k$.

The next main result of this paper handles the tensor products of two extensions fields, which will be used to generate new and original examples of regular tensor products of extension fields. It is worthwhile noting that this result falls beyond the scope of Andr\'e-Quillen homology (since purely inseparable field extensions are not geometrically regular).

%%%%%%%%%%%%%%%%%%%%%%%%%%%%%%%%%%%%%%%%%%%%%%%%%%%%%%%%%%%%%%%%%%
%%%%%%%%%%%%%%%%%%%%%%%%%%%%%%%%%%%%%%%%%%%%%%%%%%%%%%%%%%%%%%%%%%
\begin{thm}\label{Reg:2}
Let $K$ and $L$ be two extension fields of $k$ such that $K\otimes_kL$ is Noetherian. Assume that $K=K_{s}K_{i}$ and let $K_{i}=k(S)$ for some generating subset $S$ of $K_{i}$. Then the following assertions are equivalent:
\begin{enumerate}
\item $K\otimes_kL$ is regular;

\item $K_i\otimes_kL$ is a domain;

\item $K_i\otimes_kL$ is a field;

\item $[k(S'):k]=[L(S'):L]$ for each finite subset $S'$ of $S$;

\item $K_i\cap L(S')=k(S')$ for each finite subset $S'$ of $S$.
\end{enumerate}
\end{thm}

%%%%%%%%%%%%%%%%%%%%%
%%%%%%%%%%%%%%%%%%%%%
\begin{proof} Let $p:=\charac(k)$. The theorem easily holds when $p=0$ (in which case $k$ is perfect). Next, assume $p\geq 1$. Since $K_{s}$ is a separable extension of $k$, $K_{s}\otimes_{k}K_{i}$ is reduced \cite[Chap. III, \S 15, Theorem 39]{ZS}. Further, since $K_{i}$ is algebraic over $k$, $K_{s}\otimes_{k}K_{i}$ is zero-dimensional \cite[Theorem 3.1]{S2} and hence a von Neumann regular ring \cite[Ex. 22, p. 64]{Kap}. By \cite[Proposition 2(c)]{V}, $K_{s}\otimes_{k}K_{i}$ has one unique minimal prime ideal. It follows that  $K_{s}\otimes_{k}K_{i}$ is local and therefore a field. Now,
consider the surjective ring homomorphism $\varphi: K_{s}\otimes_{k}K_{i} \rightarrow K_{s}(K_{i})$, given on generators of $K_{s}\otimes_{k}K_{i}$ by $a\otimes b\mapsto ab$ (as $K_{s}$ and $K_{i}$ may be contained in a common field). So $\varphi$ is an isomorphism; that is, $K_{s}\otimes_{k}K_{i}\cong K_{s}K_{i}=K$. By Lemma~\ref{Reg:1}, $K\otimes_kL\cong K_{s}\otimes_{k}(K_{i}\otimes_kL)$ is regular if and only if $K_{i}\otimes_kL$ is regular.
Hence, for the rest of the proof,  we may suppose that $K$ is a purely inseparable algebraic extension field of $k$ (i.e., $K=K_i$) with  $\charac(k)=p\not=0$. Same arguments as above yield $K\otimes_kL$ is a zero-dimensional local ring and, therefore, (i) $\Rightarrow$ (ii) $\Rightarrow$ (iii) $\Rightarrow$ (i). Moreover, the assumption ``$K\otimes_kL$ is a domain" is equivalent to saying that ``$K$ and $L$ are linearly disjoint over $k$," as mentioned above. So that we get (ii) $\Leftrightarrow$ (iv) by \cite[Proposition 5 (a), p. V.13]{B4-7} and (ii) $\Rightarrow$ (v) by \cite[p. V.13]{B4-7} and via the isomorphism $K\otimes_kL\cong K\otimes_{k(S')}\big(k(S')\otimes_kL\big)$
for each finite subset $S'$ of $S$.

(v) $\Rightarrow$ (iii) Let $x\in S$ and let $p^{m} = [k(x):k]$ with $m$ an integer $\geq 0$. Then $a:=x^{p^{m}}\in k$. We wish to show that $k(x)\otimes_kL$ is a field. We may assume $x\notin k$. By (v),  $x^{p^r}\not\in K\cap L=k$ for each
positive integer $r<m$. Therefore, $x\in \overline L\setminus L$, where $\overline L$ denotes
an algebraic closure of $L$, forcing $(X^{p^m}-a)$ $(=(X^{p^r}-x^{p^r})^{p^{m-r}}$
for each positive integer $r<m$) to be irreducible in $L[X]$. It follows that
 $$k(x)\otimes_kL\cong \dfrac{k[X]}{(X^{p^m}-a)}\otimes_kL\cong\dfrac{L[X]}{(X^{p^m}-a)}\cong L[x] = L(x)$$
where $X$ denotes an indeterminate over $\overline{L}$. So $k(x)\otimes_kL$ is a field. Next, let $x_1,...,x_n\in S$. We
 have $$k(x_1,...,x_n)\otimes_kL\cong k(x_1,...,x_n)\otimes_{k(x_1,...,x_{n-1})}\big(k(x_1,...,x_{n-1})\otimes_kL\big).$$ By
 induction on $n$, $k(x_1,...,x_{n-1})\otimes_kL\cong L(x_1,...,x_{n-1})$ is a
 field and, by (v), we get $$k(x_1,...,x_n)\cap L(x_1,...,x_{n-1})\subseteq K\cap L(x_1,...,x_{n-1})=k(x_1,...,x_{n-1})$$ so
 that $$k(x_1,...,x_n)\cap L(x_1,...,x_{n-1})=k(x_1,...,x_{n-1}).$$ Hence, the first step yields
 $$k(x_1,...,x_n)\otimes_kL\cong k(x_1, ..., x_{n-1})(x_n)\otimes_{k(x_1,...,x_{n-1})} L(x_1,...,x_{n-1})$$ is a field. Let $\Delta$ denote the set of all finite subset $S'$ of $S$ and observe that $$K\otimes _kL=\lim_{\substack{\rightarrow\\ S'\in \Delta}}k(S')\otimes _kL.$$
 Thus, $k(S')\otimes_kL$ is a field, for each $S'\in \Delta$, and so is their direct limit $K\otimes_kL$, establishing (iii) and completing the proof of the theorem.
\end{proof}

One can use Theorem~\ref{Reg:2}(v) to build new examples of regular tensor products of fields, as illustrated by the next example.

%%%%%%%%%%%%%%%%%%%%%%%%%%%%%%%%%%%%%%%%%%%%%%%%%%%%%%%%%%%%%%%%%%
%%%%%%%%%%%%%%%%%%%%%%%%%%%%%%%%%%%%%%%%%%%%%%%%%%%%%%%%%%%%%%%%%%
\begin{example}\label{Reg:2.1.1}\rm
Let  $p$ be a prime element of $\Z$ and let $y_1, y_2, ..., y_m, x_1, x_2, ..., x_n, ...$ be indeterminates over ${\dfrac{\Z}{p\Z}}$. Let
\[\begin{array}{lcl}
k   &:= &{\dfrac{\Z}{p\Z}}\Big(y_1^{p},y_2^{p^2},...,y_m^{p^m},x_1^{p},x_2^{p^2},...,x_n^{p^n},...\Big),\\
K   &:= &k(x_1,x_2,...,x_n,...),\\
L   &:= &k(y_1,y_2,...,y_m).
\end{array}\]
Then $K\otimes_kL$ is a regular ring.

Indeed, notice that $K$ and $L$ are purely inseparable extension
fields of $k$ with $[L:k]<\infty$. Also, we have
\[\begin{array}{lcl}
 K  & = &{\dfrac{\Z}{p\Z}}\Big(y_1^p,y_2^{p^2},...,y_m^{p^m}, x_1,x_2,..,x_n,...\Big),\\
 L  & = &{\dfrac{\Z}{p\Z}}\Big(y_1,y_2,...,y_m, x_1^{p},x_2^{p^2},...,x_n^{p^n},...\Big).
 \end{array}\]
Next, let $x_{i_1},x_{i_2},...,x_{i_r}$ be a finite subset of $\{x_1,x_2,...,x_n,...\}$. Then
\[\begin{array}{lcl}
K\cap L(x_{i_1},x_{i_2},...,x_{i_r})    &=  &{\dfrac{\Z}{p\Z}}\Big(x_{i_1},x_{i_2},...,x_{i_r},x_1^{p},x_2^{p^2},...,x_n^{p^n},...,y_1^p,y_2^{p^2},...,y_m^{p^m}\Big)\\
                                        & = &k(x_{i_1},x_{i_2},...,x_{i_r}).
\end{array}\]
Hence, by Theorem~\ref{Reg:2}(v), $K\otimes_kL$ is regular, as desired.
\end{example}

As a consequence of Theorem~\ref{Reg:2}(v), under the Noetherianity assumption, separability rises as a necessary (and sufficient) condition for regularity in the special case where $K=L$ as shown in the next corollary. It also refines \cite[Exercice 28, Chap. 8, p. 98]{BAC8-9} and links regularity of $K\otimes_kK$ to the projectivity of $K$ as a $K\otimes_kK$-module when $K$ is a finitely generated extension field of $k$ \cite[Theorem 7.10]{CE}.

%%%%%%%%%%%%%%%%%%%%%%%%%%%%%%%%%%%%%%%%%%%%%%%%%%%%%%%%%%%%%%%%%%
%%%%%%%%%%%%%%%%%%%%%%%%%%%%%%%%%%%%%%%%%%%%%%%%%%%%%%%%%%%%%%%%%%
\begin{corollary}\label{Reg:2.2}
Let $K$ be an extension field of $k$. The following assertions are equivalent:
\begin{enumerate}
\item $K\otimes_kK$ is regular;

\item $K\otimes_kK$ is Noetherian and $K$ is separable over $k$;

\item $K$ is a finitely generated separable extension field of $k$;

\item $K\otimes_kL$ is regular for each extension field $L$ of $k$;

\item $K$ is a finitely generated extension field of $k$ and a projective $K\otimes_kK$-module.
\end{enumerate}
\end{corollary}

%%%%%%%%%%%%%%%%%%%%%
%%%%%%%%%%%%%%%%%%%%%
\begin{proof}
(i) $\Rightarrow$ (ii) Assume that $K\otimes_kK$
is regular. Then $K\otimes_kK$ is Noetherian, so that $K$ is
finitely generated over $k$. We claim that $K\otimes_EK$ is regular
for any extension field $E$ of $k$ contained in $K$. In effect, let
$E$ be a field extension of $k$ contained in $K$. Then
$$\begin{array}{lll}
K\otimes_kK &\cong  &K\otimes_E(E\otimes_kK)\\
            &\cong  &K\otimes_E(K\otimes_kE)\\
            &\cong  &(K\otimes_EK)\otimes_kE\ \mbox { (Cf. \cite[Ex. 2.15, p. 27]{AM}).}
\end{array}$$
It follows, by \cite[Theorem 23.7]{Mat} and by localization, that $K\otimes_EK$ is
regular, establishing the claim. Now, let $B$ be a finite
transcendence basis of $K$ over $k$ and let $E$ be the algebraic
separable closure of $k(B)$ in $K$. Then, via the above claim,
$K\otimes_EK$ is regular and $K$ is purely inseparable over $E$.
By Theorem~\ref{Reg:2}(v), $K=E$. It follows that $K$ is separable over
$k$, as desired.

(ii) $\Rightarrow$ (iii) is handled by \cite[Corollary 3.6]{Fer} or \cite[Theorem 11]{V} as mentioned above, (iii) $\Rightarrow$ (iv) follows from \cite[Lemma 6.7.4.1]{Gr}, (iv) $\Rightarrow$ (i) is trivial, and  (iii) $\Leftrightarrow$ (v) is a particular case of \cite[Theorem 7.10]{CE}, completing the proof of the corollary.
\end{proof}

One can use Theorem~\ref{Reg:2}(v) or Corollary~\ref{Reg:2.2} to build (zero-dimensional Noetherian local) tensor products of fields that are locally complete intersection but not regular, as shown below.

%%%%%%%%%%%%%%%%%%%%%%%%%%%%%%%%%%%%%%%%%%%%%%%%%%%%%%%%%%%%%%%%%%
%%%%%%%%%%%%%%%%%%%%%%%%%%%%%%%%%%%%%%%%%%%%%%%%%%%%%%%%%%%%%%%%%%
\begin{example}\label{Reg:2.3}\rm
Let $k\subsetneqq K\subseteq L$ be extension fields such that $K$ is purely inseparable over $k$ and $K\otimes_kL$ is Noetherian. Then $K\otimes_kL$ is a locally complete intersection ring \cite[Proposition 5(a)]{TY} which is not regular by Theorem~\ref{Reg:2}(v) (or Corollary~\ref{Reg:2.2}). For instance, for any prime $p$, one may simply take $$k:= \dfrac{\Z}{p\Z}(x^{p})\ \mbox{and}\ K=L:=\dfrac{\Z}{p\Z}(x)$$
where $x$ is an indeterminate over $\dfrac{\Z}{p\Z}$.
\end{example}

The next result handles the (algebraic) separable case featuring a slight generalization of \cite[Proposition 8]{V}. Recall, for convenience, that if $K$ is a separable extension of $k$, then $K\otimes_kL$ is always reduced for any extension field $L$ of $k$ \cite[Chap. III, \S 15, Theorem 39]{ZS}.

%%%%%%%%%%%%%%%%%%%%%%%%%%%%%%%%%%%%%%%%%%%%%%%%%%%%%%%%%%%%%%%%%%
%%%%%%%%%%%%%%%%%%%%%%%%%%%%%%%%%%%%%%%%%%%%%%%%%%%%%%%%%%%%%%%%%%
\begin{corollary}\label{Reg:2.4}
Let $K$ and $L$ be two extension fields of $k$ such that $K\otimes_kL$ is Noetherian. Assume that $K$ is algebraic over $k$. Then the following assertions are equivalent:
\begin{enumerate}
\item $K\otimes_kL$ is (von Neumann) regular;
\item $K\otimes_kL$ is reduced;
\item $K\otimes_kL$ is a finite product of fields.
\end{enumerate}
If, in addition, $K$ is separable and  $L$ is Galois over $k$ such that $K,L$ are contained in an algebraic closure of $k$, then the above are equivalent to:
\begin{enumerate}
\item[(iv)] $n:=[K\cap L:k]<\infty$.
\end{enumerate}
Moreover, $K\otimes_kL$ is isomorphic to the product of $n$ copies of the field $K(L)$.
\end{corollary}

%%%%%%%%%%%%%%%%%%%%%
%%%%%%%%%%%%%%%%%%%%%
\begin{proof}
By \cite[Theorem 3.1]{S2}, $\dim(K\otimes_kL)=0$. Recall at this point that a
zero-dimensional Noetherian ring is regular if and only if it is von Neumann regular.  So a combination of \cite[Theorem 164]{Kap},
\cite[Ex. 22, p. 64]{Kap}, and \cite[Lemma 0]{V} yields (i) $\Leftrightarrow$ (ii) $\Leftrightarrow$ (iii). The last two statements are handled by \cite[Proposition 8]{V}.
\end{proof}

Next, we give an illustrative example for this corollary.

%%%%%%%%%%%%%%%%%%%%%%%%%%%%%%%%%%%%%%%%%%%%%%%%%%%%%%%%%%%%%%%%%%
%%%%%%%%%%%%%%%%%%%%%%%%%%%%%%%%%%%%%%%%%%%%%%%%%%%%%%%%%%%%%%%%%%
\begin{example}\label{Reg:2.5}\rm
Let $(p_{j})_{j\geq1}$ denote the sequence of all prime numbers. Let $$X:=\{i,e^{\frac{2i\pi}{3}}\}\cup\{\sqrt{p_{j}}\mid j\ \mbox{odd}\}\ \mbox{and}\ Y:=\{i\}\cup\{\sqrt{p_{j}}\mid j\ \mbox{even}\}.$$ Clearly, $\Q(X)$ (resp., $\Q(Y)$) is an infinite algebraic separable non-normal (resp., Galois) extension field of $\Q$ and hence by Corollary~\ref{Reg:2.4}
$$\Q(X)\otimes \Q(Y)\cong \Q(i, e^{\frac{2i\pi}{3}}, \sqrt{2}, \sqrt{3}, ...)\times \Q(i, e^{\frac{2i\pi}{3}}, \sqrt{2}, \sqrt{3}, ...)$$ is a non-trivial zero-dimensional regular ring.
\end{example}

Next, we move to the general case, where we discuss the correlation between $A\otimes_kB$ and its fiber rings when subject to regularity. Let $A$ and $B$ be two $k$-algebras. By identifying $A$ and $B$ with their canonical images in $A\otimes_kB$, one can view $A\otimes_kB$ as a free (hence faithfully flat) extension of $A$ and $B$. This very fact lies behind the known transfers of regularity between $A\otimes_kB$ and its fiber rings over the prime ideals of $A$ or $B$. The next result collects these transfer results along with a slight generalization of \cite[Theorem 6(c)]{TY}. We also provide an example, via Theorem~\ref{Reg:2}, for the non-reversibility in general of the implications involved. For this purpose, we first make the following definition.

%%%%%%%%%%%%%%%%%%%%%%%%%%%%%%%%%%%%%%%%%%%%%%%%%%%%%%%%%%%%%%%%%%
%%%%%%%%%%%%%%%%%%%%%%%%%%%%%%%%%%%%%%%%%%%%%%%%%%%%%%%%%%%%%%%%%%
\begin{definition}\rm
A $k$-algebra $R$ is said to be residually separable, if $\kappa_{R}(P)$ is separable over $k$ for each $P\in\Spec(R)$, where $\kappa_{R}(P)$ denotes the residue field of $R_{P}$.
\end{definition}

It is easily seen that a field $k$ is perfect if and only if every $k$-algebra is residually separable. More examples of residually separable $k$-algebras are readily available through localizations of polynomial rings or pullback constructions \cite{BG,BR}. For instance, let $x$ be an indeterminate over $k$ and $K\subseteq L$ two separable extension fields of $k$. Let $$R:=L[x]_{(x)}\ \mbox{and}\ S:=K+xL[x]_{(x)}.$$ Note that the extensions $$k\subseteq K\subseteq L\subseteq L(x)=\qf(R)=\qf(S)$$ are separable by Mac Lane's Criterion and transitivity of separability. So that $R$ and $S$ are residually separable $k$-algebras.

%%%%%%%%%%%%%%%%%%%%%%%%%%%%%%%%%%%%%%%%%%%%%%%%%%%%%%%%%%%%%%%%%%
%%%%%%%%%%%%%%%%%%%%%%%%%%%%%%%%%%%%%%%%%%%%%%%%%%%%%%%%%%%%%%%%%%
\begin{thm}\label{Reg:3}
Let $A$ and $B$ be two $k$-algebras such that $A\otimes_{k}B$ is Noetherian. Consider the following assertions:
\begin{enumerate}
\item $A$, $B$, and $\kappa_{A}(P)\otimes_{k}\kappa_{B}(Q)$ are regular $\forall\ (P,Q)\in\Spec(A)\times\Spec(B)$;

\item $B$ and $A\otimes_{k}\kappa_{B}(Q)$ are regular $\forall\ Q\in\Spec(B)$;

\item $A$ and $\kappa_{A}(P)\otimes_{k}B$ are regular $\forall\ P\in\Spec(A)$;

\item $A\otimes_{k}B$ is regular;

\item $A$ and $B$ are regular.
\end{enumerate}
Then (i) $\Rightarrow$ (ii) (resp., (iii)) $\Rightarrow$ (iv) $\Rightarrow$ (v). If $A$ (or $B$) is residually separable, then all assertions are equivalent.
\end{thm}

%%%%%%%%%%%%%%%%%%%%%
%%%%%%%%%%%%%%%%%%%%%
\begin{proof}
The first statement is a combination of Corollary 2 and Corollary 4 as well as the proof of Theorem 6 in \cite{TY}.

Next, suppose that $A$ or $B$ is residually separable. Then $\kappa_{A}(P)\otimes_{k}\kappa_{B}(Q)$ is always regular by Lemma~\ref{Reg:1} for any $P\in\Spec(A)$ and $Q\in\Spec(B)$; and, hence, so are $\kappa_{A}(P)\otimes_{k}B$ and $A\otimes_{k}\kappa_{B}(Q)$. Moreover, recall that Noetherianity carries over to   $\kappa_{A}(P)\otimes_{k}\kappa_{B}(Q)$ via localization of the general fact that  if $I$ and $J$ are proper ideals of $A$ and $B$, respectively, then $$\dfrac {A\otimes_kB}{I\otimes_kB+A\otimes_kJ}\cong \dfrac AI \otimes_k\dfrac BJ.$$ Thus, the five assertions in the theorem collapse to: ``$A\otimes_{k}B$ is regular if and only if $A$ and $B$ are regular."
\end{proof}

The above implications are not reversible in general, as shown by the next example. This example shows also that the separable assumption in Lemma~\ref{Reg:1} is sufficient but not necessary and it does not hold, in general, for purely inseparable extensions.

%%%%%%%%%%%%%%%%%%%%%%%%%%%%%%%%%%%%%%%%%%%%%%%%%%%%%%%%%%%%%%%%%%
%%%%%%%%%%%%%%%%%%%%%%%%%%%%%%%%%%%%%%%%%%%%%%%%%%%%%%%%%%%%%%%%%%
\begin{example}\label{Reg:4}\rm
Let $K$ be a purely inseparable extension field of $k$ with $\charac(k)=p\not= 0$ and let $u\in K$ with $p^{e}:=[k(u) : k]$ for some $e\geq 2$. Then $a:=u^{p^{e}}\in k$. Let $x$ be an indeterminate over $k$, $r\in\{1,\cdots, e-1\}$, and $A:=k[x]_{(x^{p^{e-r}}-a)}$. Then:
\begin{enumerate}
\item $A$ is local regular with maximal ideal $\m:=(x^{p^{e-r}}-a)A$.
\item $k(u)\otimes_{k} A$ is regular.
\item $k(u)\otimes_{k} \dfrac{A}{\m}$ is not regular.
\end{enumerate}

Indeed, notice that $(x^{p^{e-r}}-a)$ is a prime ideal of $k[x]$ and, hence, $\m$ is the maximal ideal of $A$, since $\dfrac{A}{\m}\cong \dfrac{k[x]}{(x^{p^{e-r}}-a)}\cong k(u^{p^{r}})$. Moreover, $k(u)\otimes_{k} A\cong S^{-1}k(u)[x]$ is a regular ring, where $S:=k[x]\setminus (x^{p^{e-r}}-a)$. This proves (i) and (ii). However, $k(u)\otimes_{k} \dfrac{A}{\m}\cong k(u)\otimes_{k} k(u^{p^{r}})$ is not regular, by Theorem~\ref{Reg:2}(v), since $k\not=k(u)\cap k(u^{p^{r}})=k(u^{p^{r}})$, proving (iii).
\end{example}

The assumption ``$A$ (or $B$) is residually separable" in Theorem~\ref{Reg:3} is not necessary, as shown by the following example.

%%%%%%%%%%%%%%%%%%%%%%%%%%%%%%%%%%%%%%%%%%%%%%%%%%%%%%%%%%%%%%%%%%
%%%%%%%%%%%%%%%%%%%%%%%%%%%%%%%%%%%%%%%%%%%%%%%%%%%%%%%%%%%%%%%%%%
\begin{example}\label{Reg:5}\rm

Let $k$, $K$, and $L$ be defined as in Example~\ref{Reg:2.1.1} and $x,y$ two indeterminates over $k$. Let
$$\begin{array}{lclcll}
A   &:= & K[x]_{(x)}   &=  & K+\m_A & \mbox{with } \m_A:=xA\\
B   &:= & L[y]_{(y)}   &=  & L+\m_B & \mbox{with } \m_B:=yB
\end{array}$$
Then $A$ and $B$ are regular
local $k$-algebras which are not residually separable over $k$ (since $K$ and $L$ are purely inseparable over $k$ as seen in Example~\ref{Reg:2.1.1}). Moreover, $A\otimes_{k} B$ is Noetherian (in fact, regular via localization) and ${\dfrac{A}{\m_A}\otimes_k\dfrac{B}{\m_B}}\cong K\otimes_kL$ is a regular ring. Consequently, $A$ and $B$ satisfy all assertions of Theorem~\ref{Reg:3}, as desired.
\end{example}

The next example illustrates the slight improvement (of \cite[Theorem 6(c)]{TY}) featured in the last statement of Theorem~\ref{Reg:3}. Namely, we provide original examples where $k$ is an arbitrary field, $A,B$ are regular $k$-algebras with $A\otimes_{k} B$ Noetherian and $A$ is residually separable over $k$.

%%%%%%%%%%%%%%%%%%%%%%%%%%%%%%%%%%%%%%%%%%%%%%%%%%%%%%%%%%%%%%%%%%
%%%%%%%%%%%%%%%%%%%%%%%%%%%%%%%%%%%%%%%%%%%%%%%%%%%%%%%%%%%%%%%%%%
\begin{example}\label{Reg:6}\rm
Let $k$ be an arbitrary field, $K$ any separable extension field of $k$, and $x,y,t$ three indeterminates over $k$. Consider the
$K$-algebra homomorphism $$\varphi:K[x,y]\rightarrow K[[t]]$$ defined by $\varphi (x)=t$ and $\varphi (y)=s:={\sum_{n\geq
1}}t^{n!}$. Since $s$ is known to be transcendental over $K(t)$,
$\varphi$ is injective. This induces the following embedding of fields $$\overline{\varphi}:K(x,y)\rightarrow K((t)).$$ It is
easy to check that $A:=\overline{\varphi}^{-1}(K[[t]])$ is a discrete rank-one valuation overring of
$K[x,y]$ and that $A=K+\m$ with $\m=xA$. Then, $A$ is a residually
separable regular ring. Now, let $B$ be any regular ring such that $A\otimes_kB$ is Noetherian. For instance, one may choose $B$ to be any finitely generated regular
$k$-algebra or any (purely inseparable) finitely generated extension field of $k$. By Theorem~\ref{Reg:3}, $A\otimes_kB$ is a
regular ring.
\end{example}

It is worthwhile noticing that, in most examples, the non-regularity  was ensured by the negation of ``$K_{i}\cap L=k$." One might wonder if this weak property may generate the condition (v) of Theorem~\ref{Reg:2}; namely, let $K$ be a finite dimensional purely inseparable extension field of $k$ and let $L$ be an extension field of $k$. Do we have: $K\cap L=k\Leftrightarrow K\otimes_kL$ regular? The answer is negative as shown by the next example.

%%%%%%%%%%%%%%%%%%%%%%%%%%%%%%%%%%%%%%%%%%%%%%%%%%%%%%%%%%%%%%%%%%
%%%%%%%%%%%%%%%%%%%%%%%%%%%%%%%%%%%%%%%%%%%%%%%%%%%%%%%%%%%%%%%%%%
\begin{example}\label{Reg:7}\rm
Let $x,y,z$ be three indeterminates over ${\dfrac{\Z}{2\Z}}$. Let
$$\begin{array}{lcl}
k   &:= &{\dfrac{\Z}{2\Z}}\Big(x^4,y^4\Big),\\
K   &:= & k(x^2,y^2) = {\dfrac{\Z}{2\Z}}\Big (x^2,y^2\Big),\\
L   &:= &k\big(x^2(y^2+z),z\big) = {\dfrac{\Z}{2\Z}}\Big(x^4,x^2(y^2+z),z\Big).
\end{array}$$
Then $K\cap L=k$ and $K\otimes_kL$ is not a regular ring.

Indeed, clearly, $K$ is a purely inseparable extension field  of $k$.  Further, note that $\{1,x^2\}$ is a basis
of $K$ over $k(y^2)$ and, as $(x^2(y^2+z))^2\in k(z)$,
$\{1,x^2(y^2+z)\}$ is a basis of $L$ over $k(z)$. Let $f\in K\cap L$. So there exist
$g_0,g_1\in k(y^2)$ and $f_0,f_1\in k(z)$ such that
$$\left\{
\begin{array}{lcl}
f   &=  &g_0+g_1x^2\\
    &=  &f_0+f_1x^2(y^2+z).
\end{array}
\right.$$
As $(x^2)^2\in k(y^2,z)$ and $x^2\not\in
k(y^2,z)={\dfrac {\Z}{2\Z}}\Big
(x^4,y^2,z\Big )$, then $\{1,x^2\}$ is, as well, a basis of
$k(x^2,y^2,z)$ over $k(y^2,z)$. It follows that $f_0=g_0$ and
$f_1(y^2+z)=g_1$. Hence, $f_0\in k(z)\cap k(y^2)=k$. Moreover,
observe that $\{1,y^2\}$ is a basis of $k(y^2,z)$ over $k(z)$ and of
$k(y^2)$ over $k$. Hence, as $g_1=f_1z+f_1y^2$ and $g_1\in k(y^2)$,
we get $f_1z\in k$, so that $f_1=0$. Consequently, $f\in k$ and
therefore $K\cap L=k$, as claimed.

Now, $L(x^2)=k(x^2,y^2,z)=K(z)$. Hence $K\cap L(x^2)=K\neq k(x^2)$.
Then, by Theorem~\ref{Reg:2}(v), $K\otimes_kL$ is not regular, as desired.
\end{example}
\end{section}
%\newpage

%%%%%%%%%%%%%%%%%%%%%%%%%%%%%%%%%%%%%%%%%%%%%%%%%%%%%%%%%%%%%%%%%%%%%%%%%%%%%%%%%%%%%%%%%%%%%%%%%%%%%%%%%%%%%%%%%%%%%%%%%%%%%%%%%%%%%%%%%%%%%%
%%%%%%%%%%%%%%%%%%%%%%%%%%%%%%%%%%%%%%%%%%%%%%%%%%%%%%%%%%%%%%%%%%%%%%%%%%%%%%%%%%%%%%%%%%%%%%%%%%%%%%%%%%%%%%%%%%%%%%%%%%%%%%%%%%%%%%%%%%%%%%

%%%%%%%%%%%%%%%%%%%%%%%%%%%%%%%%%%%%%%%%%%%%%%%%%%%%%%%%%%%%%%%%%%%%%%%%%%%%%%%%%%%%%%%%%%%%%%%%%%%%%%%%%%%%%%%%%%%%%%%%%%%%%%%%%%%%%%%%%%%%%%
%%%%%%%%%%%%%%%%%%%%%%%%%%%%%%%%%%%%%%%%%%%%%%%%%%%%%%%%%%%%%%%%%%%%%%%%%%%%%%%%%%%%%%%%%%%%%%%%%%%%%%%%%%%%%%%%%%%%%%%%%%%%%%%%%%%%%%%%%%%%%%
%%%%%%%%%%%%%%%%%%%%%%%%%%%%%%%%%%%%%%%%%%%%%%%%%%%%%%%%%%%%%%%%%%%%%%%%%%%%%%%%%%%%%%%%%%%%%%%%%%%%%%%%%%%%%%%%%%%%%%%%%%%%%%%%%%%%%%%%%%%%%%
%%%%%%%%%%%%%%%%%%%%%%%%%%%%%%%%%%%%%%%%%%%%%%%%%%%%%%%%%%%%%%%%%%%%%%%%%%%%%%%%%%%%%%%%%%%%%%%%%%%%%%%%%%%%%%%%%%%%%%%%%%%%%%%%%%%%%%%%%%%%%%

\end{document}